\documentclass[final,11pt]{elsarticle}

\usepackage{verbatim,a4wide}

\usepackage{amsmath}
\usepackage{amssymb}
\usepackage{amsthm}
\usepackage{bm}

\usepackage[cp1250]{inputenc}
\usepackage[OT4]{fontenc}
\usepackage[english]{babel}

\numberwithin{equation}{section}

\newtheorem{theorem}{Theorem}[section]
\newtheorem{lemma}[theorem]{Lemma}
\newdefinition{remark}[theorem]{Remark}
\newdefinition{corollary}[theorem]{Corollary}
\newdefinition{definition}[theorem]{Definition}
\newdefinition{problem}[theorem]{Problem}
\newdefinition{example}[theorem]{Example}

\newcommand{\hyper}[5]{\,{}_{#1}F_{#2}\!\left(%
            \begin{array}{cc}{\displaystyle{#3}}\\[0.25ex]%
            {\displaystyle{#4}} \end{array}\bigg|\,{\displaystyle{#5}}%
            \right)}

\journal{Applied Mathematics and Computation} 

\begin{document}

\begin{frontmatter}

\title{Differential-recurrence properties of dual Bernstein polynomials}

\author{Filip Chudy}
\ead{Filip.Chudy@cs.uni.wroc.pl}

\author{Pawe{\l} Wo\'{z}ny\corref{cor}}
\ead{Pawel.Wozny@cs.uni.wroc.pl}
\cortext[cor]{Corresponding author. Fax {+}48 71 3757801}
\address{Institute of Computer Science, University of Wroc{\l}aw,
         ul.~Joliot-Curie 15, 50-383 Wroc{\l}aw, Poland}

\begin{abstract}
New differential-recurrence properties of dual Bernstein polynomials are given
which follow from relations between dual Bernstein and orthogonal Hahn and 
Jacobi polynomials. Using these results, a fourth-order differential equation
satisfied by dual Bernstein polynomials has been constructed. Also, 
a fourth-order recurrence relation for these polynomials has been obtained; 
this result may be useful in the efficient solution of some computational
problems.
\end{abstract}


\end{frontmatter}

\section{Introduction}                                  \label{S:Introduction}

Dual Bernstein polynomials associated with the Legendre inner product were
introduced by Ciesielski in 1987 \cite{ZC1987}. Their properties and 
generalizations were studied, e.g., by J\"{u}ttler \cite{BJ1998}, Rababah and
Al-Natour \cite{RN2007,RN2008}, as well as by Lewanowicz and Woźny
\cite{LW2006,LW2011,WL2009}. It is worth noticing that dual Bernstein
polynomials introduced in~\cite{LW2006}, which are associated with the shifted
Jacobi inner product, have recently found many applications in numerical
analysis and computer graphics (curve intersection using B\'{e}zier clipping,
degree reduction and merging of B\'{e}zier curves, polynomial approximation of
rational B\'{e}zier curves, etc.). Note that skillful use of these polynomials
often results in less costly algorithms which solve some computational problems
(see
\cite{BJ2007,GLW2016,GLW2017,LWK2012,LZLW2009,SN1990,WGL2015,WL2009}).

The main purpose of this article is to give new properties of dual Bernstein
polynomials considered in~\cite{LW2006}. Namely, we derive some
differential-recurrence relations which allow us to construct a differential 
equation and a recurrence relation for these polynomials.   

The paper is organized as follows. Section~\ref{S:DualBernsteinPoly} contains
definitions, notation and important properties of dual Bernstein polynomials
obtained in~\cite{LW2006}. Next, in Section~\ref{S:DiffRecRel}, we present
new results which imply: i) the fourth-order differential equation with 
polynomial coefficients (see~\S\ref{S:DiffEqDualBer}); ii) the recurrence
relation of order four (see~\S\ref{S:RecRelDualBer}), both of which are
satisfied by dual Bernstein polynomials. The latter result may be useful in 
finding the efficient solution of some computational tasks, e.g., fast
evaluation of dual Bernstein polynomials and their linear combinations or
integrals involving these dual polynomials (see~\S\ref{S:Applications}).

\section{Dual Bernstein polynomials}               \label{S:DualBernsteinPoly}

The \textit{generalized hypergeometric function} (see, e.g., 
\cite[\S2.1]{AAR1999}) is defined by
$$
\hyper{p}{q}{a_1,\ldots, a_p}{b_1,\ldots,b_q}{x}
          :=\sum_{l=0}^{\infty}\frac{(a_1)_l\ldots(a_p)_l}
                                    {(b_1)_l\ldots(b_q)_l}\cdot\frac{x^l}{l!},
$$
where $p,q\in\mathbb N$, $a_i\in\mathbb C$ $(i=1,2,\ldots,p)$,
$b_j\in\mathbb C$ $(j=1,2,\ldots,q)$, $x\in\mathbb C$, and $(c)_l$ 
$(c\in\mathbb C;\ l\in\mathbb N)$ denotes the \textit{Pochhammer symbol},
$$
(c)_0:=1,\qquad (c)_l:=c(c+1)\ldots(c+l-1)\quad (l\geq 1).
$$
Notice that if one of the parameters $a_i$ is equal to $-k$ $(k\in\mathbb N)$ 
then the generalized hypergeometric function is a polynomial in $x$ of degree 
at most $k$.

For $\alpha,\beta>-1$, let us introduce the inner product
$\left<\cdot,\cdot\right>_{\alpha,\beta}$ by
\begin{equation}\label{E:InnerProd}
\left<f,g\right>_{\alpha,\beta}:=
             \int_{0}^{1}(1-x)^\alpha x^\beta f(x)g(x)\,\mbox{d}x.
\end{equation}
             
Recall that \textit{shifted Jacobi polynomials} 
$R_k^{(\alpha,\beta)}$ (cf., e.g., \cite[\S1.8]{KS1998}),
\begin{equation}\label{E:JacobiP}
R_k^{(\alpha,\beta)}(x):=\frac{(\alpha+1)_k}{k!}
\hyper{2}{1}{-k,\, k+\alpha+\beta+1}{\alpha+1}{1-x}\qquad (k=0,1,\ldots),
\end{equation}
are orthogonal with respect to the inner product~\eqref{E:InnerProd}, i.e.,
$$
\left<R^{(\alpha,\beta)}_k,R^{(\alpha,\beta)}_l\right>_{\alpha,\beta}=
\delta_{kl}h_k\qquad (k,l\in\mathbb N),
$$
where $\delta_{kl}$ is the \textit{Kronecker delta} ($\delta_{kl}=0$ for 
$k\neq l$ and $\delta_{kk}=1$) and
$$
h_k:=K\,\frac{(\alpha+1)_k(\beta+1)_k}
             {k!(2k/\sigma+1)(\sigma)_{k}}\qquad (k=0,1,\ldots)
$$
with $\sigma:=\alpha+\beta+1$, 
                    $K:={\Gamma(\alpha+1)\Gamma(\beta+1)}/{\Gamma(\sigma+1)}$.

Shifted Jacobi polynomials satisfy the second-order differential equation
with polynomial coefficients of the form (cf.~\cite[Eq.~(1.8.5)]{KS1998})
\begin{equation}\label{E:DiffEqJacobi}
\bm{L}^{(\alpha,\beta)}R^{(\alpha,\beta)}_k(x)=
        \lambda_k^{(\alpha,\beta)}R^{(\alpha,\beta)}_k(x)\qquad (k=0,1,\ldots),
\end{equation}
where
$$
\bm{L}^{(\alpha,\beta)}:=x(x-1)\bm{D}^2+
  \tfrac{1}{2}\left(\alpha-\beta+(\sigma+1)(2x-1)\right)\bm{D},\qquad
\lambda_k^{(\alpha,\beta)}:=k(k+\sigma),  
$$
and\label{P:DiffOper}
$\displaystyle \bm{D}:=\frac{\mbox{d}}{\mbox{d}x}$ is a differentiation 
operator with respect to the variable $x$. 

It is well known that (cf.~\cite[p.~117]{AAR1999}) 
\begin{equation}\label{E:JacobiSym}
R^{(\alpha,\beta)}_k(x)=(-1)^kR^{(\beta,\alpha)}_k(1-x).
\end{equation}

Moreover, we also use the second family of orthogonal polynomials,
namely \textit{Hahn polynomials},
\begin{equation}\label{E:HahnP}
Q_k(x;\alpha,\beta;N):=\hyper{3}{2}{-k,\,k+\alpha+\beta+1,\, -x}
                                   {\alpha+1,\, -N}{1}
\qquad (k=0,1,\ldots,N;\ N\in\mathbb N)                                   
\end{equation}
(see, e.g., \cite[\S1.5]{KS1998}). 

Hahn polynomials satisfy the second-order difference equation with polynomial
coefficients of the form
\begin{equation}\label{E:DiffEqHahn}
{\cal L}^{(\alpha,\beta,N)}_x Q_k(x;\alpha,\beta;N)=\lambda_k^{(\alpha,\beta)}
Q_k(x;\alpha,\beta;N)\qquad (k=0,1,\ldots),
\end{equation}
where
\begin{equation}\label{E:OperL}
{\cal L}^{(\alpha,\beta,N)}_x f(x):=a(x)f(x+1)-c(x)f(x)+b(x)f(x-1),
\end{equation}
and
$$
a(x):=(x-N)(x+\alpha+1),\qquad b(x):=x(x-\beta-N-1),\qquad c(x):=a(x)+b(x).
$$
See, e.g., \cite[Eq.~(1.5.5)]{KS1998}.

Let $\Pi_n$ $(n\in\mathbb N)$ denote the set of polynomials of degree at most
$n$. \textit{Bernstein basis polynomials} $B^n_i$ are given by
\begin{equation}\label{E:BernPoly}
B^n_i(x):=\binom{n}{i}x^i(1-x)^{n-i}\qquad (i=0,1,\ldots,n;\ n\in\mathbb N).
\end{equation}
One can easily check that polynomials $B^n_0, B^n_1, \ldots, B^n_n$ form 
a basis of the space $\Pi_n$.

Bernstein basis polynomials~\eqref{E:BernPoly} have many applications in
approximation theory, numerical analysis, as well as in computer aided
geometric design (see, e.g., books \cite{Bustamante2017}, \cite{Farin2002} 
and papers cited therein). In view of their applications in computer graphics
and numerical analysis, the so-called \textit{dual Bernstein polynomials}
have become quite popular.

\begin{definition}[{\cite[\S5]{LW2006}}]
\textit{Dual Bernstein polynomials of degree $n$},
$$
D^n_0(x;\alpha,\beta),\, D^n_1(x;\alpha,\beta),\,\ldots,\,
D^n_n(x;\alpha,\beta)\in\Pi_n,
$$
are defined so that the following conditions hold:
$$
\left<B^n_i,D^n_j(\cdot;\alpha,\beta)\right>_{\alpha,\beta}=\delta_{ij}
\qquad (i,j=0,1,\ldots,n)
$$
(cf.~\eqref{E:InnerProd}).
\end{definition}

For the properties and applications of dual Bernstein polynomials
$D^n_i(x;\alpha,\beta)$, see
\cite{BJ2007,GLW2016,GLW2017,LW2006,LWK2012,LZLW2009,SN1990,WGL2015,WL2009}. 
Note that in the case $\alpha=\beta=0$ these polynomials were defined earlier 
by Ciesielski in~\cite{ZC1987}.

\begin{remark}
We adopt the convention that  $D^n_i(x;\alpha,\beta):=0$ for $i<0$ or $i>n$.
\end{remark}


Dual Bernstein polynomials, Hahn polynomials and shifted Jacobi polynomials
are related in the following way \cite[Theorem 5.2)]{LW2006}:
\begin{equation}\label{E:DualB-Jacobi}
D^n_i(x;\alpha,\beta)=
      K^{-1}\sum_{k=0}^{n}
         (-1)^k\frac{(2k/\sigma+1)(\sigma)_k}
                    {(\alpha+1)_k}Q_k(i;\beta,\alpha;n)R^{(\alpha,\beta)}_k(x)
\qquad (0\leq i\leq n).                    
\end{equation}
  
Note that
\begin{equation}\label{E:DualBerSym}
D^n_i(x;\alpha,\beta)=D^n_{n-i}(1-x;\beta,\alpha)\qquad (i=0,1,\ldots,n)
\end{equation}
(see \cite[Corollary 5.3]{LW2006}).

The polynomial $D^n_i(x;\alpha,\beta)$ can be expressed as \textit{a short}
linear combination of $\min(i,n-i)+1$ shifted Jacobi polynomials with shifted
parameters: 
\begin{eqnarray}
\label{E:DualBerShort}
D^n_i(x;\alpha,\beta)&=&\frac{(-1)^{n-i}(\sigma+1)_{n}}
                             {K\,(\alpha+1)_{n-i}(\beta+1)_i}
                                  \sum_{k=0}^{i}\frac{(-i)_k}{(-n)_k}\,
				         R^{(\alpha,\beta+k+1)}_{n-k}(x),\\
\nonumber				         
D^n_{n-i}(x;\alpha,\beta)&=&\frac{(-1)^{i}(\sigma+1)_{n}}
                                 {K\,(\alpha+1)_{i}(\beta+1)_{n-i}}
                                    \sum_{k=0}^{i}(-1)^k\frac{(-i)_k}{(-n)_k}
                                             \,R^{(\alpha+k+1,\beta)}_{n-k}(x),
\end{eqnarray}
where $i=0,1,\ldots,n$. See~\cite[Corollary 5.4]{LW2006}. 

\section{Differential-recurrence relations}               \label{S:DiffRecRel}

Let us first find the representation of the polynomial $D^n_i(x;\alpha,\beta)$ 
in the basis $(1-x)^j$ $(j=0,1,\ldots,n)$. By using~\eqref{E:JacobiP} 
in~\eqref{E:DualBerShort} and doing some algebra, we obtain 
\begin{eqnarray}\nonumber
D^n_i(x;\alpha,\beta)&=&\frac{(-1)^{n-i}(\sigma+1)_{n}}
                             {K\,(\alpha+1)_{n-i}(\beta+1)_i}
                                  \sum_{k=0}^{i}\frac{(-i)_k}{(-n)_k}
                                    \frac{(\alpha+1)_{n-k}}{(n-k)!}
                           \hyper{2}{1}{k-n,\, n+\sigma+1}{\alpha+1}{1-x}\\
&=&\label{E:DualBer_Power}
A^{(\alpha,\beta)}_{ni}\frac{(\alpha+1)_n}{(n+1)!}
                \sum_{j=0}^{n}B^{(\alpha,\beta)}_{nj}
                  \hyper{3}{2}{j-n,\,-i,\,1}{-n,\,-n-\alpha}{1}\cdot (1-x)^j,
\end{eqnarray}
where  
\begin{equation}\label{E:Def-A-B}
A^{(\alpha,\beta)}_{ni}:=\frac{(-1)^{n-i}(n+1)(\sigma+1)_{n}}
                              {K\,(\alpha+1)_{n-i}(\beta+1)_i},\qquad
B^{(\alpha,\beta)}_{nj}:=\frac{(-n)_j(n+\sigma+1)_j}{j!(\alpha+1)_j}.                              
\end{equation}

Let us define
$$
F(i,j):=\hyper{3}{2}{j-n,\,-i,\,1}{-n,\,-n-\alpha}{1}
\qquad (i,j=0,1,\ldots,n).
$$
Using the Zeilberger algorithm \cite[\S6]{PWZ}, one can prove the following
lemma.

\begin{lemma}\label{L:RecRelF_ij}
Quantities $F(i,j)$ satisfy the first-order non-homogeneous recurrence relation
of the form
\begin{equation}\label{E:RecRelF_ij}
(i-n)(n-i+\alpha)F(i+1,j)-(i+1)(n+j-i+\alpha+1)F(i,j)=-(n+1)(n+\alpha+1),
\end{equation}
where $0\leq i,j\leq n$ and we adopt the convention that $F(n+1,j):=0$. 
\end{lemma}

Lemma~\ref{L:RecRelF_ij} allows us to give the first of the mentioned 
differential-recurrence relations for dual Bernstein polynomials.

\begin{theorem}\label{T:DiffRec-I}
For $i=0,1,\ldots,n$, the following formula holds:
\begin{eqnarray}\nonumber
\lefteqn{\Big((1-x)\bm{D}-(n-i+\alpha+1)\bm{I}\Big)D^n_i(x;\alpha,\beta)}\\
&&\label{E:DiffRec-I}
\hspace{1.5cm}
=\frac{(i-n)(i+\beta+1)}{i+1}D^n_{i+1}(x;\alpha,\beta)
-A^{(\alpha,\beta)}_{ni}\frac{n+\alpha+1}{i+1}R^{(\alpha,\beta+1)}_n(x),
\end{eqnarray}
where $\displaystyle \bm{D}:=\frac{\mbox{d}}{\mbox{d}x}$
(cf.~p.~\ref{P:DiffOper}), and $\bm{I}$ is the identity operator.
\end{theorem}
\begin{proof}
We add up the recurrence relation~\eqref{E:RecRelF_ij}, multiplied by
$B^{(\alpha,\beta)}_{nj}(1-x)^j$, over all $0\leq j\leq n$ and take into 
account that
\begin{eqnarray*}
&&
R^{(\alpha,\beta+1)}_n(x)=
       \frac{(\alpha+1)_n}{n!}\sum_{j=0}^{n}B^{(\alpha,\beta)}_{nj}(1-x)^j,\\
&&
\bm{D}D^n_i(x;\alpha,\beta)=-A^{(\alpha,\beta)}_{ni}\frac{(\alpha+1)_n}{(n+1)!}
                  \sum_{j=1}^{n}B^{(\alpha,\beta)}_{nj}F(i,j)\cdot j(1-x)^{j-1}
\end{eqnarray*}
(cf.~\eqref{E:Def-A-B}).
\end{proof}


Another relation for $D^n_i(x;\alpha,\beta)$ can be found by applying symmetry
relations~\eqref{E:JacobiSym} and~\eqref{E:DualBerSym} in~\eqref{E:DiffRec-I}.

\begin{theorem}\label{T:DiffRec-II}
For $i=0,1,\ldots,n$, we have
\begin{eqnarray}\nonumber
\lefteqn{\Big(x\bm{D}+(i+\beta+1)\bm{I}\Big)D^n_i(x;\alpha,\beta)}\\
&&\label{E:DiffRec-II}
\hspace{1.5cm}
=\frac{i(n-i+\alpha+1)}{n-i+1}D^n_{i-1}(x;\alpha,\beta)+
A^{(\alpha,\beta)}_{ni}\frac{n+\beta+1}{n-i+1}R^{(\alpha+1,\beta)}_n(x).
\end{eqnarray}
\end{theorem}

The next differential-recurrence relation is more complicated. It relates
the second and first derivative of $D^n_i(x;\alpha,\beta)$ with the 
polynomials $D^n_{i-1}(x;\alpha,\beta)$, $D^n_i(x;\alpha,\beta)$, 
$D^n_{i+1}(x;\alpha,\beta)$.

\begin{theorem}\label{T:DiffRec-III}
The following relation holds:
\begin{eqnarray}\label{E:DiffRec-III}
\lefteqn{\Big(x(x-1)\bm{D}^2+\tfrac{1}{2}(\alpha-\beta+(\sigma+1)(2x-1))
                                       \bm{D}\Big)D^n_i(x;\alpha,\beta)}\\
&&\nonumber\hspace{1cm}
=(i-n)(i+\beta+1)D^n_{i+1}(x;\alpha,\beta)
+i(i-\alpha-n-1)D^n_{i-1}(x;\alpha,\beta)\\
&&\nonumber\hspace{5cm}
-(i(i-\alpha-n-1)+(i-n)(i+\beta+1))D^n_{i}(x;\alpha,\beta),
\end{eqnarray}
where $i=0,1,\ldots,n$.
\end{theorem}
\begin{proof}
We use the representation~\eqref{E:DualB-Jacobi} of dual Bernstein polynomials,
the differential equation~\eqref{E:DiffEqJacobi} for shifted Jacobi 
polynomials, as well as the difference equation~\eqref{E:DiffEqHahn} satisfied
by Hahn polynomials.

Observe that   
\begin{eqnarray*}
\bm{L}^{(\alpha,\beta)}D^n_{i}(x;\alpha,\beta)
      &=&K^{-1}\sum_{k=0}^{n}
         (-1)^k\frac{(2k/\sigma+1)(\sigma)_k}
                    {(\alpha+1)_k}Q_k(i;\beta,\alpha;n)\cdot
                     \lambda_k^{(\alpha,\beta)}R^{(\alpha,\beta)}_k(x)\\
      &=&K^{-1}\sum_{k=0}^{n}
         (-1)^k\frac{(2k/\sigma+1)(\sigma)_k}
                    {(\alpha+1)_k}R^{(\alpha,\beta)}_k(x)\cdot
                     \lambda_k^{(\beta,\alpha)} Q_k(i;\beta,\alpha;n)\\                     
&=&{\cal L}^{(\beta,\alpha,n)}_i D^n_{i}(x;\alpha,\beta).                    
\end{eqnarray*}
\end{proof}

\section{Differential equation}                        \label{S:DiffEqDualBer}

Using the new properties of dual Bernstein polynomials given in
Section~\ref{S:DiffRecRel}, one can construct the differential equation for
$D^n_i(x;\alpha,\beta)$.

\begin{theorem}\label{T:DualBer-NonHomo-Diff-Eq}
Dual Bernstein polynomials satisfy the second-order non-homogeneous 
differential equation with polynomial coefficients of the form
\begin{equation}\label{E:DualBer-NonHomo-Diff-Eq}
\bm{M}^{(\alpha,\beta)}_{ni}D^n_i(x;\alpha,\beta)=
          (n+\sigma+1)A^{(\alpha,\beta)}_{ni}R^{(\alpha+1,\beta+1)}_n(x),
\end{equation}
where
$$
\bm{M}^{(\alpha,\beta)}_{ni}:=
     x(x-1)\bm{D}^2+\Big((n+\sigma+3)x-i-\beta-2\Big)\bm{D}+(n+\sigma+1)\bm{I}.
$$
\end{theorem}
\begin{proof}
By substituting the expressions for $D^n_{i+1}(x;\alpha,\beta)$ and
$D^n_{i-1}(x;\alpha,\beta)$ determined by~\eqref{E:DiffRec-I} 
and~\eqref{E:DiffRec-II}, respectively, into equation~\eqref{E:DiffRec-III}, 
we obtain
$$
\bm{M}^{(\alpha,\beta)}_{ni}D^n_i(x;\alpha,\beta)=
        A^{(\alpha,\beta)}_{ni}
                     \left((n+\alpha+1)R^{(\alpha,\beta+1)}_n(x)+
                               (n+\beta+1)R^{(\alpha+1,\beta)}_n(x)\right).
$$
To complete the proof, observe that 
$$
(n+\alpha+1)R^{(\alpha,\beta+1)}_n(x)+
                               (n+\beta+1)R^{(\alpha+1,\beta)}_n(x)=
(n+\sigma+1)R^{(\alpha+1,\beta+1)}_n(x),                               
$$                               
which follows from~\eqref{E:JacobiP} after some algebra.
\end{proof}

Notice that by applying the second-order differential operator
$$
\bm{N}^{(\alpha,\beta)}_{ni}:=
\bm{L}^{(\alpha+1,\beta+1)}-\lambda_n^{(\alpha+1,\beta+1)}\bm{I}
$$ 
(cf.~\eqref{E:DiffEqJacobi}) to both sides of
Eq.~\eqref{E:DualBer-NonHomo-Diff-Eq}, we obtain the homogeneous differential
equation for dual Bernstein polynomials. 

\begin{corollary}\label{C:DualBer-Diff-Eq}
Dual Bernstein polynomials $D^n_i(x;\alpha,\beta)$ $(i=0,1,\ldots,n)$ satisfy
the fourth-order differential equation with polynomial coefficients of the form
\begin{equation}\label{E:DualBer-Diff-Eq}
\bm{Q}_4D^n_i(x;\alpha,\beta)\equiv
                 \bm{N}^{(\alpha,\beta)}_{ni}
                         \bm{M}^{(\alpha,\beta)}_{ni}D^n_i(x;\alpha,\beta)=0.
\end{equation}
\end{corollary}

Observe that the operator $\bm{Q}_4$ is a composition of two second-order
differential operators. For the reader's convenience, we give also
the explicit form of the differential equation~\eqref{E:DualBer-Diff-Eq}:
$$
\sum_{j=0}^{4}w_j(x)\bm{D}^jD^n_i(x;\alpha,\beta)=0,
$$
where
\begin{eqnarray*}
&&
w_4(x):=x^2(x-1)^2,\qquad 
w_3(x):=x(x-1)[(n+2\sigma+10)x-i-2\beta-6],\\
&&
w_2(x):=[(n+\sigma+3)(\sigma-n+7)+\sigma+3]x^2\\
&&
\hspace{2cm}
+[(n-1)^2+\alpha n-2\beta-(\sigma+3)(i+2\beta+8)-5]x+(\beta+2)(i+\beta+3),\\
&&
w_1(x):=-(n+\sigma+2)[(n^2+(n-2)(\sigma+3))x+(2-n)(i+\beta+2)-2i],\\
&&
w_0 (x):=-n(n+\sigma+1)_2.
\end{eqnarray*}

\section{Recurrence relation}                          \label{S:RecRelDualBer}                   

In~\cite[Theorem 5.1]{LW2006}, the following recurrence relation, which
connects dual Bernstein polynomials of degrees $n+1$ and $n$, as well as the
shifted Jacobi polynomial of degree $n+1$, was given:
\begin{equation}\label{E:RecRel-I}
D^{n+1}_i(x;\alpha,\beta)=\left(1-\dfrac{i}{n+1}\right)
   \,D^{n}_{i}(x;\alpha,\beta)+\frac{i}{n+1}\,D^{n}_{i-1}(x;\alpha,\beta)+
                         C^{(\alpha,\beta)}_{ni}R^{(\alpha,\beta)}_{n+1}(x),
\end{equation}
where $0\le i\le n+1$, and 
$$
C^{(\alpha,\beta)}_{ni}:=(-1)^{n-i+1}\frac{(2n+\sigma+2)(\sigma+1)_n}
                             {K(\alpha+1)_{n-i+1}(\beta+1)_{i}}.
$$
Let us mention that the case $\alpha=\beta=0$ of this relation was found 
earlier by Ciesielski~in~\cite{ZC1987}.

Now, using the results given in Section~\ref{S:DiffRecRel}, we show that it is
possible to construct a~homogeneous recurrence relation connecting five
consecutive (with respect to $i$) dual Bernstein polynomials of the same
degree~$n$.  

Let ${\cal E}^m$ be the $m$th \textit{shift operator} acting on the variable
$i$ in the following way:
$$
{\cal E}^m z_i:=z_{i+m}\qquad (m\in\mathbb Z).
$$ 
For the sake of simplicity, we write ${\cal I}:={\cal E}^0$ and 
${\cal E}:={\cal E}^1$.

For example, the operator ${\cal L}^{(\alpha,\beta,N)}_{i}$
(cf.~\eqref{E:OperL} and the proof of Theorem~\ref{T:DiffRec-III}) can be
written as:
$$
{\cal L}^{(\alpha,\beta,N)}_{i}=a(i){\cal E}-c(i){\cal I}+b(i){\cal E}^{-1}.
$$

The following theorem holds.

\begin{theorem}\label{T:DualBer-NonHomo-Rec-Rel}
Dual Bernstein polynomials satisfy the second-order non-homogeneous recurrence
relation of the form
\begin{equation}\label{E:DualBer-NonHomo-Rec-Rel}
{\cal M}^{(\alpha,\beta,n)}_{i}D^n_i(x;\alpha,\beta)=
                                                   G_{ni}^{(\alpha,\beta)}(x),
\end{equation}
where $i=0,1,\ldots,n$, and
\begin{eqnarray*}
&&{\cal M}^{(\alpha,\beta,n)}_{i}:=(i)_2(n-i+\alpha+1)(x-1){\cal E}^{-1}
-(n-i)_2(i+\beta+1)x{\cal E}\\
&&\hspace{4.75cm}+(i+1)(n-i+1)[(i+\beta+1)(1-x)+(n-i+\alpha+1)x]{\cal I},\\
&&G_{ni}^{(\alpha,\beta)}(x):=
A_{ni}^{(\alpha,\beta)}\Big((i+1)(n+\beta+1)(1-x)R_n^{(\alpha+1,\beta)}(x)\\
&&\hspace{4.75cm}+(n-i+1)(n+\alpha+1)xR_n^{(\alpha,\beta+1)}(x)\Big).
\end{eqnarray*}
\end{theorem}
\begin{proof}
The recurrence~\eqref{E:DualBer-NonHomo-Rec-Rel} can be obtained in
the following way: we subtract the relation~\eqref{E:DiffRec-I}, 
multiplied by $x$, from the relation~\eqref{E:DiffRec-II}, multiplied by $1-x$. 
\end{proof}

Notice that the quantity
$H(i):=\left(A_{ni}^{(\alpha,\beta)}\right)^{-1}G_{ni}^{(\alpha,\beta)}(x)$ 
is a polynomial of the first degree in variable $i$. Thus we have
$({\cal E}-{\cal I})^2H(i)=0$. By applying the operator 
$$
{\cal N}^{(\alpha,\beta,n)}_{i}:={\cal E}^{-1}
        ({\cal E}-{\cal I})^2\left(A_{ni}^{(\alpha,\beta)}\right)^{-1}{\cal I}
$$
to both sides of the equation~\eqref{E:DualBer-NonHomo-Rec-Rel}, we
obtain a fourth-order homogeneous recurrence relation for the dual
Bernstein polynomials. 

\begin{corollary}\label{C:DualBer-Rec-Rel}
Dual Bernstein polynomials 
satisfy the fourth-order recurrence relation of the form
\begin{equation}\label{E:DualBer-Rec-Rel-a}
{\cal Q}_4D^n_i(x;\alpha,\beta)\equiv
    {\cal N}^{(\alpha,\beta,n)}_{i}{\cal M}^{(\alpha,\beta,n)}_{i}
                              D^n_i(x;\alpha,\beta)=0\qquad (0\leq i\leq n).
\end{equation}
\end{corollary}

Let us stress that the operator ${\cal Q}_4$ is a composition of two 
second-order difference operators. Below, we also give the explicit form
of the simplified recurrence relation~\eqref{E:DualBer-Rec-Rel-a}:
\begin{equation}\label{E:DualBer-Rec-Rel-b}
\sum_{j=-2}^{2}v_j(i)D^n_{i+j}(x;\alpha,\beta)=0,
\end{equation}
where
\begin{eqnarray*}
&&v_{-2}(i):=(1-x)(i-1)_2(n-i+\alpha)_3,\\
&&v_{-1}(i):=-i(n-i+\alpha)_2\{(i+\beta)(n-3i)\\
&&\hspace{2cm}+[n(n-3i+\alpha-\beta+4)+i(4i-\alpha+3\beta-4)+2(\alpha+2)]x\},\\
&&v_0(i):=(i+\beta)(n-i+\alpha)[z(i)x+(i+1)(i+\beta+1)(3i-2n)],\\
&&v_1(i):=(i-n)(i+\beta)_2\{(i+2)(i+\beta+2)\\
&&\hspace{2cm}-[n(2n-5i+2\alpha)+i(4i-3\alpha+\beta+4)+2(\beta+2)]x\},\\
&&v_2(i):=x(i+\beta)(i+\beta+1)_2(n-i-1)_2,
\end{eqnarray*}
and 
$z(i):=-6i^3+3(3n+\alpha-\beta)i^2-[n(5n-6\beta)+(4n+3)\sigma+3]i+
                                                 n[(n+1)(n+\alpha+1)+2\beta+2]$.

\section{Applications}                                  \label{S:Applications}                   

Now, we point out some possible applications of the obtained recurrence
relation. Let us consider the following task.  

\begin{problem}\label{P:Problem1}
Let us fix numbers: $n\in\mathbb N$, $x\in\mathbb C$ and
$\alpha,\beta>-1$. Consider the problem of computing the values 
$$
D^n_i(x;\alpha,\beta)
$$
for all $i=0,1,\ldots,n$. 
\end{problem}

An efficient solution of this problem gives us, e.g., the fast method of
evaluating the polynomial
\begin{equation}\label{E:DulBer-Poly-Form}
d(x):=\sum_{i=0}^{n}d_iD^n_i(x;\alpha,\beta),
\end{equation}
where coefficients $d_0,d_1,\ldots,d_n$ are given. Notice that such
representation plays a crucial role in the algorithm for merging of B\'{e}zier
curves which has been recently proposed in~\cite{WGL2015}. 

On the other hand, in many applications, such as least-square approximation in
B\'{e}zier form (cf.~\cite{LW2011}, \cite{LWK2012}) or numerical solving of
boundary value problems (cf., e.g., report~\cite{GW2018}) or fractional partial 
differential equations (see \cite{JBJ2017}, \cite{JJBB2017} and papers cited
therein), it is necessary to compute the collection of integrals of the form
$$
I_k:=\int_{0}^{1}(1-x)^\alpha x^\beta f(x)D^n_k(x;\alpha,\beta)\,\mbox{d}x
$$
for all $k=0,1,\ldots,n$ and a given function $f$. Recall that the main reason 
is that a polynomial
$$
p^\ast_n(x):=\sum_{k=0}^{n}I_kB^n_k(x)
$$
minimizes the value of the least-square error
$$
\int_{0}^{1}(1-x)^\alpha x^\beta (f(x)-p_n(x))^2\mbox{d}x\qquad (p_n\in\Pi_n).
$$
The numerical approximations of the integrals $I_0, I_1,\ldots, I_n$ 
involving the dual Bernstein polynomials can be computed, for example, by
quadrature rules (see, e.g., \cite[\S5]{DB}). It also requires the fast 
evaluation of polynomials $D^n_0(x;\alpha,\beta), D^n_1(x;\alpha,\beta),\ldots,
D^n_n(x;\alpha,\beta)$ in many \textit{nodes}.

The solutions of Problem~\ref{P:Problem1} which use the 
representations~\eqref{E:DualB-Jacobi}, \eqref{E:DualBerShort}
or~\eqref{E:DualBer_Power} of dual Bernstein polynomials, or the recurrence
relation~\eqref{E:RecRel-I} satisfied by these polynomials, have too high
computational complexity (notice that one has to compute also shifted 
Jacobi and/or Hahn polynomials, cf.~\eqref{E:JacobiP} and \eqref{E:HahnP}). 

Observe that it is more efficient to use the recurrence
relation~\eqref{E:DualBer-Rec-Rel-b} which is not explicitly related to shifted
Jacobi and Hahn polynomials. This recurrence allows us to solve the 
problem with the computational complexity $O(n)$. For details,
see \cite[\S7 and \S10.2]{Wimp1984}.

Horner's rule (see, e.g., \cite[Eq.~(1.2.2)]{DB}) for evaluating the $n$th
degree polynomial given in the power basis also has the computational 
complexity $O(n)$. Taking into account that the dual Bernstein basis is much
more complicated than the power basis, the algorithms based on the
recurrence~\eqref{E:RecRel-I} for evaluating $D^n_i(x;\alpha,\beta)$ or 
a polynomial given in the form~\eqref{E:DulBer-Poly-Form} seem to be
interesting.

To show the efficiency of the new recurrence relation for dual Bernstein
polynomials, let us present the following numerical example. The results
have been obtained on a computer with \texttt{Intel Core i5-661 3.33Hz}
processor and \texttt{8GB} of \texttt{RAM}, using computer algebra 
system \textsf{Maple{\small\texttrademark}~8}. 

\begin{example}
For $n=10,15,20$ and $\alpha=\beta=0$ (\textit{Legendre's} case),
$\alpha=\beta=-0.5$ (\textit{Chebyshev's} case) and $\alpha=-0.33, \beta=5.66$
(\textit{non-standard} case), the values of dual polynomials
$D^n_i(x_k;\alpha,\beta)$ at all the points 
$x_r:=\tfrac{k}{M}$ $(0\leq k\leq M;\;M=100)$ and for all $i=0,1,\ldots,n$ have 
been computed by recurrence relations~\eqref{E:RecRel-I} (computational
complexity $O(Mn^2)$) and~\eqref{E:DualBer-Rec-Rel-b} (computational complexity
$O(Mn)$). Both methods give results of similar numerical quality. However the
algorithm using the new recurrence relation~\eqref{E:DualBer-Rec-Rel-b}
is significantly faster. See Table~\ref{T:Table1}.
\begin{table}
\renewcommand{\arraystretch}{1.75}
\begin{center}
\begin{tabular}{ll|cc|cc}
&& \multicolumn{2}{c}{Recurrence~\eqref{E:RecRel-I}} &
\multicolumn{2}{|c}{Recurrence~\eqref{E:DualBer-Rec-Rel-b}} \\
&& time & error & time & error \\ \hline
$n=10$ & $\alpha=\beta=0$ & $2.453$ & $0.40\cdot 10^{-31}$ & 
                            $0.688$ & $0.40\cdot 10^{-31}$ \\
       & $\alpha=\beta=-0.5$ & $2.750$ & $0.20\cdot 10^{-28}$ & 
                               $0.953$ & $0.71\cdot 10^{-26}$ \\
       & $\alpha=-0.33,\;\beta=5.66$ & $3.845$ & $0.25\cdot 10^{-24}$ & 
                                       $1.563$ & $0.33\cdot 10^{-24}$ \\ \hline
$n=15$ & $\alpha=\beta=0$ & $5.984$ & $0.35\cdot 10^{-25}$ & 
                            $0.937$ & $0.41\cdot 10^{-23}$ \\
       & $\alpha=\beta=-0.5$ & $8.000$ & $0.19\cdot 10^{-22}$ 
                             & $1.485$ & $0.11\cdot 10^{-22}$ \\
       & $\alpha=-0.33,\;\beta=5.66$ & $12.391$ & $0.13\cdot 10^{-21}$ & 
                                       $2.781$ & $0.44\cdot 10^{-20}$ \\ \hline
$n=20$ & $\alpha=\beta=0$ & $12.327$ & $0.18\cdot 10^{-19}$ & 
                            $1.329$ & $0.26\cdot 10^{-19}$ \\
       & $\alpha=\beta=-0.5$ & $17.734$ & $0.72\cdot 10^{-19}$ & 
                               $2.125$ & $0.17\cdot 10^{-18}$ \\
       & $\alpha=-0.33,\;\beta=5.66$ & $27.797$ & $0.41\cdot 10^{-19}$ & 
                                       $4.735$ & $0.90\cdot 10^{-19}$  
\end{tabular}
\end{center}
\renewcommand{\arraystretch}{1}
\caption{Results of numerical experiments (total time in seconds and
maximum error for $M=100$).}\label{T:Table1}
\end{table}
\end{example}

%



\bibliographystyle{elsart-num-sort} 
\biboptions{compress}
\bibliography{dualB-diff-eq-rev}


\end{document}